\documentclass[12pt]{article}
\usepackage{amssymb,amsmath}

\newtheorem{thm}     {Theorem}[section]
\newtheorem{prop}    [thm]{Proposition}

\newtheorem{lemma}   [thm]{Lemma}

\newcommand{\proof} {\noindent{\bf Proof. }}

\newcommand{\B}{\mathbb B}
\newcommand{\C}{\mathbb C}
\newcommand{\D}{\mathbb D}
\def\H{\mathbb H}

\newcommand{\R}{\mathbb R}

\def\Re{{\rm Re\,}}
\def\Im{{\rm Im\,}}
\def\bar{\overline}
\def\Area{{\rm Area}}

\def\const{{\rm const}}
\def\<{\langle}
\def\>{\rangle}

\begin{document}

\title{Pseudoholomorphic discs and symplectic \\
structures in Hilbert space}
\author{Alexandre Sukhov{*} and Alexander Tumanov{**}}
\date{}
\maketitle

{\small

* Universit\'e des Sciences et Technologies de Lille, Laboratoire
Paul Painlev\'e,
U.F.R. de
Math\'e-matique, 59655 Villeneuve d'Ascq, Cedex, France. The author is partially supported by Labex CEMPI,
e-mail address: sukhov@math.univ-lille1.fr

** University of Illinois, Department of Mathematics
1409 West Green Street, Urbana, IL 61801, USA, e-mail address: tumanov@illinois.edu
}
\bigskip

Abstract. We develop the theory of $J$-holomorphic discs in Hilbert spaces with almost complex structures. As an aplication, we prove a version of Gromov's symplectic non-squeezing theorem
for Hilbert spaces. It can be applied to short-time symplectic flows of a wide class of Hamiltonian PDEs.
\bigskip



\tableofcontents

\section{Introduction}

In his fundamental work, Gromov \cite{Gr} proposed a new approach to the symplectic geometry based on the theory of pseudoholomorphic curves in almost complex manifolds. Every symplectic manifold $(M,\omega)$  admits many almost complex structures $J$ in a natural co-relation with the symplectic structure $\omega$. A pseudoholomorphic (or $J$-holomorphic, or $J$-complex) curve in $(M,J)$ is a holomorphic map from a Riemann surface $X$ to $M$. In view of the uniformization theorems this is not surprising that the following two choices of $X$ are of major importance: the unit disc $\D$ and the Riemann sphere $S^2$. The corresponding $J$-holomorphic maps are called respectively $J$-holomorphic discs and $J$-holomorphic spheres. The spheres appear naturally in the case where $M$ is a compact manifold while the discs are appropriate when $M$ is a manifold with boundary. In the latter case usually some boundary conditions are imposed making the families of such discs (moduli spaces) finite dimensional. In turns out that moduli spaces of $J$-holomorphic discs or spheres reflect some important topological properties of the underlying symplectic structure $\omega$ and in some cases allow to describe efficiently symplectic invariants. Thus, Gromov's approach became a powerful tool of symplectic geometry.

Usual symplectic structures and Hamiltonians on smooth manifolds arise from the  classical mechanics and dynamical systems with finite-dimensional phase spaces. On the other hand, the  important classes of Hamiltonian PDEs of the mathematical physics, such as the (non-linear) Schr\"odinger equation, the Korteweg-de Vries equation and others, have a Hamiltonian nature too, but the corresponding symplectic structures and flows are defined on suitable Hilbert spaces (see for instance \cite{Ku2}). Usually these Hilbert spaces are Sobolev spaces, in which the existence and regularity of solutions of (the Cauchy problem for) Hamiltonian PDEs are established. In the finite-dimensional case, one of the most striking results of Gromov is the symplectic non-squeezing theorem describing quantitatively the topological rigidity of symplectic transformations. The first results of this type for various classes of Hamiltonian PDEs were obtained by Kuksin \cite{Ku1,Ku2} and later extended in the work of Bourgain \cite{Bou1,Bou2}, Colliander, Keel,  Staffilani, Takaoka, and Tao \cite{Tao}, and Roum\'egoux \cite {Rou}. Their approach is based on approximation of a symplectic flow in Hilbert space by finite-dimensional symplectic flows which reduces the situation to Gromov's theorem. This method requires a version of such approximation theory separately for each class of Hamiltonian PDEs. Recently Abbondandolo and Majer \cite{Ab} made a step toward a general version of the infinite-dimensional non-squeezing using the theory of symplectic capacities. In view of the above work it seems appropriate to extend Gromov's theory of pseudoholomorphic curves to the case of the Hilbert spaces equipped with symplectic and almost complex structures. The goal of the present paper is to develop an analog of Gromov's theory of pseudoholomorphic curves directly on Hilbert's spaces making it available for the study of symplectic flows of a wide class of Hamiltonian PDEs.

If $(M,J)$ is an almost complex manifold of complex dimension $n$, then pseudoholomorphic curves in $M$  are solutions of quasilinear ellipitic systems of PDEs with two independent and 2n dependent variables. Locally such a system can be viewed as a perturbation of the usual $\bar\partial$-system. Global problems lead to elliptic equations of Beltrami type which in general are large deformations of the usual Cauchy-Riemann equations. This allows us to apply methods of complex analysis and PDEs. The infinite-dimensional case requires the analysis of quasi-linear Beltrami type equations for Hilbert space-valued functions. The main method for solving the Beltrami equation for scalar-valued functions is based on the theory of singular integrals and the non-linear analysis, especially the fixed point theory. We use this approach in the present paper. We point out that our approach is rather different from Gromov's approach based on the compactness and transversality theory for pseudoholomorphic curves. The extension of Gromov's theory to infinite dimension remains a difficult open problem. As an application of our methods we prove in the last section a version of the non-squeezing theorem in Hilbert space for symplectic maps with small ``anti-holomorphic terms''. In particular, it can be applied to short-time symplectic flows or small perturbations of holomorphic symplectic flows associated with  Hamiltonian PDEs.

Our main results are contained in Section 4, 5 and 6. In Section 4 we establish the local existence and regularity of pseudoholomorphic curves in Hilbert spaces with almost complex structures. In Section 5 we
consider boundary value problems for pseudoholomorphic curves in Hilbert space. Our approach is based on methods developed in our previous work in the finite-dimensional case \cite{CoSuTu,SuTu3,SuTu1}. Using suitable singular integral operators related to the Cauchy integral, we reduce boundary value problems to integral equations in appropriate
function spaces and construct solutions by Schauder's fixed point theorem. Sections 2 and 3 are devoted to technical tools of our method and contain some necessary properties of almost complex structures in Hilbert spaces and some results on singular integral operators in spaces of vector functions. We present some of them in detail because of their importance for our approach and because we could not find references in the literature.

In conclusion we note that by using the methods and technical tools elaborated in the present paper, we obtain in \cite{SuTu4} a version of the non-squeezing theorem for general symplectic transformations under certain regularity and boundedness assumptions with respect to Hilbert scales (see Section 5). The proofs of the main results of the present paper are independent of \cite{SuTu4}.

\section{Almost complex structures on Hilbert spaces}

In this section we establish some basic facts on almost complex structures in Hilbert spaces. We always restrict to {\it separable} Hilbert spaces.

\subsection{Almost complex structures}

Let $V$ be a real vector space. If the dimension of $V$ is finite, we assume that it is even.  A {\it linear almost complex structure} $J$ on $V$ is a bounded linear operator $J: V \longrightarrow V$ satisfying $J^2 = -I$. Here and below, $I$ denotes the identity map or the identity matrix depending on the context.

\begin{thm}
\label{TheoEquiv}
Let $\H$ be a real Hilbert space. Then every two linear almost complex structures $J_k$, $k=1,2$, are equivalent. That is, there exists a bounded linear invertible operator $R:\H \to \H$ such that $J_1R = RJ_2$.
\end{thm}
Since all (separable and $\infty$-dimensional) complex Hilbert spaces are isometrically isomorphic, the above theorem follows by
\begin{prop}
Let $\H$ be a real Hilbert space. Let also $J$ be a linear almost complex structure on $\H$. Then there is an equivalent norm and an inner product on $\H$ making it into a complex Hilbert space.
\end{prop}
\begin{proof} Introduce a new norm $\| \bullet \|_n$ on $\H$:
$$
\| h \|_n^2 = \| h \|^2 + \| J h \|^2, \; h \in \H.
$$
Since $J$ is a bounded operator, the new norm is equivalent to the original one.
It is immediate that the new norm is $J$-invariant: $\| Jh \|_n^2 = \| h \|_n^2$ and verifies the parallelogram identity. Hence, this norm is a Hilbert space norm.
\end{proof} $\blacksquare$

Let $\H$ be a complex Hilbert space with Hermitian  scalar product $\langle \bullet, \bullet \rangle$. Fix an orthonormal basis $\{ e_j \}_{j=1}^\infty$ in $\H$ such that $Z = \sum_{j=1}^\infty Z_j e_j$ for every $Z \in \H$. Here $Z_j = x_j + iy_j= \langle Z,e_j \rangle$ are complex coordinates of $Z$. Put $\bar{Z} = \sum_{j=1}^\infty \bar{Z}_j e_j$. The {\it standard} almost complex structure $J_{st}$ on $\H$ is a real linear operator on $\H$ defined by $J_{st}(Z) = iZ$. In the case where $\H$ has a finite dimension $n$ the structure $J_{st}$  is the usual complex structure on $\C^n$; we do not specify the dimension (finite or infinite) in this notation since it will be clear from the context.

In what follows we also use the standard notation
$$
f_Z = \frac{\partial f}{\partial Z} = \frac{1}{2} \left ( \frac{\partial f}{\partial x} - i\frac{\partial f}{\partial y} \right), \,\,\,\,\,\,  f_{\bar Z} = \frac{\partial f}{\partial \bar Z}=\frac{1}{2}\left ( \frac{\partial f}{\partial x} + i\frac{\partial f}{\partial y} \right ).
$$

The {\it standard symplectic form} $\omega$ on $\H$ is a nondegenerate antisymmetric bilinear form defined by
$$\omega = \frac{i}{2}\sum_{j=1}^\infty dZ_j \wedge d\bar{Z}_j.$$
We  use the natural identification of $\H$ with its tangent space at every point.  Denote by ${\mathcal L}(\H)$  the space  of real linear bounded operators on $\H$. {\it An almost complex structure} $J$ on $\H$ is a continuous map   $J: \H \to {\mathcal L}(\H)$,   $J: \H \ni Z \to J(Z)$ such that every $J(Z)$  satisfies $J^2(Z) = -I$. Such a structure is {\it tamed}  by  $\omega$  if   $\omega(u,Ju) > 0$ for every $u \in \H$. An almost complex structure $J$ is {\it compatible} with $\omega$ if it is tamed and $\omega(Ju,Jv) = \omega(u,v)$ for all $u,v \in \H$.  Note that  $J_{st}$ is compatible with $\omega$. When the map $Z \mapsto J(Z)$ is independent of $Z$, we can identify the tangent space of $\H$ at $Z$ with $\H$ and view $J$ as a linear almost complex structure on $\H$.

A $C^1$-map $f:(\H,J') \to (\H,J)$ is called $(J',J)$-{\it holomorphic} if it satisfies {\it the Cauchy-Riemann equations}
\begin{eqnarray}
\label{CRglobal0}
J \circ df= df \circ J'.
\end{eqnarray}

When $f:\H \to \H$ is a diffeomorphism and $J'$ is an almost complex structure on $\H$, one can consider its direct image defined by
$$
f_*(J') = df \circ J' \circ df^{-1}.
$$

Of course, $f$ is $(J',f_*(J'))$-holomorphic.

Denote by  $\D = \{ \zeta \in \C : | \zeta | < 1 \}$  the unit disc in $\C$. It is equipped with the standard complex structure $J_{st}$ of $\C$. Let $J$ be an almost complex structure on $\H$. A $C^1$- map
$f:\D \to \H$ is called  a $J$-{\it complex disc} in $\H$ if it satisfies (\ref{CRglobal0}), i.e.,
\begin{eqnarray}
\label{CRglobal}
J \circ df= df \circ J_{st}.
\end{eqnarray}

A $C^1$-difeomorphism $\Phi: \Omega_1 \to \Omega_2$ between two open  subsets $\Omega_j$ in $(\H,\omega)$ is called a {\it symplectomorphism} if $\Phi^*\omega = \omega$. Here the star denotes the pull-back. For a map $Z: \D \to \H$, $Z: \zeta \mapsto Z(\zeta)$ its (symplectic) {\it area}  is defined by
\begin{eqnarray}
\label{area}
\Area(Z) = \int_\D Z^*\omega
\end{eqnarray}
similarly to the finite-dimensional case. If $Z$ is a $J$-complex disc, its symplectic area coincides with the area induced by the Riemannian metric canonically defined by $J$ and $\omega$. Hence if $Z$ is $J_{st}$-holomorphic,  (\ref{area}) represents its  area induced by the inner product of $\H$.

\subsection{Cauchy-Riemann equations}

All linear operators in this subsection are bounded. For an $\R$-linear operator $F: \H \to \H$ we denote by $F^*$  its adjoint, that is, $\Re \langle Fu,v \rangle = \Re \langle u, F^* v \rangle$. Put
\begin{eqnarray*}
\bar{F} Z = \bar{(F \bar Z)} , \,\,\, \mbox{and} \,\,\, F^{t} = \bar{ F^*}.
\end{eqnarray*}
Thus $F^t$ is the transpose of $F$.
Every $\R$-linear operator $F: \H \to \H$ has the form
$$
F u = P u + Q \bar{u},
$$
where $P$ and $Q$ are $\C$-linear operators. For brevity we write
$$
F = \{ P, Q\}.
$$
Note that
$$
F^* = \{ P^*, Q^t \}, \; F^t = \{ P^t, Q^* \}.
$$

\begin{lemma}
\label{LemAC1}
Let $F= \{ P, Q \}$. Then $F$ preserves $\omega$ i.e.
$\omega(Fu,Fv) =\omega(u,v)$ if and only if
\begin{eqnarray}
\label{identity1}
P^* P - Q^t \bar{Q} = I \;
\mbox{and}\;
P^t \bar{Q} - \bar{Q}^t P= 0.
\end{eqnarray}
\end{lemma}
\begin{proof} Consider
$$R =   \left(
\begin{array}{cll}
P & & Q\\
\bar{Q} & & \bar{P}
\end{array}
\right)$$
as a linear operator on $\H \oplus \H$. Then $F$ preserves $\omega$ if and only if $R$ preserves a bilinear form on $\H \oplus \H$ with the matrix

$$\Lambda = \left(
\begin{array}{cll}
0 & & -I\\
I & & 0
\end{array}
\right),$$
that is $R^t \Lambda R = \Lambda$. This is equivalent to (\ref{identity1}).
\end{proof} $\blacksquare$

A linear operator $F:\H \to \H$ is called a {\it linear symplectomorphism} if $F$ preserves $\omega$ and is invertible.

\begin{lemma}
\label{LemAC2}
Let $F = \{ P, Q \}$ be a linear symplectomorphism. Then $F^t$ also preserves $\omega$. Hence
\begin{eqnarray}
\label{identity2}
P P^* - Q Q^* = I\,\,\,\,\mbox{and}\,\,\,P Q^t - Q P^t = 0.
\end{eqnarray}
\end{lemma}
\begin{proof}
Since $F$ preserves $\omega$, we have
$R^t \Lambda R = \Lambda$. The operator $R$ is invertible because $F$ is. Multiplying by $R \Lambda$ from the left and by $R^{-1}\Lambda$ from the right, we obtain $R\Lambda R^t = \Lambda$.
The latter  is equivalent to (\ref{identity2}).
\end{proof} $\blacksquare$

\begin{prop}
\label{PropAC3}
If $F = \{ P, Q \}$ is a linear symplectomorphism, then
\begin{itemize}
\item[(a)] $ F^{-1} = \{ P^*, -Q^t \}$;
\item[(b)] $P$ is invertible;
\item[(c)] $\| Q \bar{P}^{-1} \| = \| Q \| ( 1 + \| Q \| ^2)^{-1/2} < 1$.
\end{itemize}
\end{prop}
\begin{proof} (a) follows by (\ref{identity1}) and (\ref{identity2}). By (\ref{identity1}) and (\ref{identity2}), spectral values of the self-adjoint operators $P P^*$ and $P^* P$ are not smaller that $1$. Then both $P^* P$ and $PP^*$ are invertible which gives (b). For (c), put $A = Q \bar{P}\,^{-1}$. We estimate $\| A \| = \| A A^* \|^{1/2}$.
By (\ref{identity1}) and (\ref{identity2}) respectively, we have $Q \bar{P}\,^{-1} =\bar{P^t}\,^{-1}Q^t$ and $Q^t (P^t)^{-1} = P^{-1}Q$. Using the latter, $A A^* = (P P^*)^{-1} Q Q^*$. Since $P P^* = I + Q Q^*$ and $Q Q^*$ is self-adjoint, by the spectral mapping theorem
$$
\| A A^* \| = \frac{\| Q Q^* \|}{1 + \| Q Q^* \|} = \frac{\| Q \|^2}{1 + \| Q \|^2}
$$
because the function $\lambda \mapsto \lambda(1+ \lambda)^{-1}$  is increasing for $\lambda>0$.
\end{proof} $\blacksquare$

Assume that $J$ is an almost complex structure tamed by $\omega$. Then $\omega(h,(J_{st} + J)h) > 0$ for all $h \neq 0$ and the operator $J_{st} + J$ is injective.
In the finite dimensional case this implies that the operator is invertible. In the Hilbert case this is not so immediate although one can show that this is also always true.  Assume that  for all $Z \in \H$ the operator
\begin{eqnarray}
\label{oper0}
(J_{st} + J)(Z)
\end{eqnarray}
 is invertible.   Then the linear operator
\begin{eqnarray}
\label{oper1}
L:= (J_{st} + J)^{-1}(J_{st} - J)
\end{eqnarray}
is defined and bounded.

\begin{lemma}
The operator (\ref{oper1}) is $\C$-antilinear.
\end{lemma}
\begin{proof} Note
\begin{eqnarray}
\label{com1}
(I + J_{st}J)^{-1} \,\,\,\mbox{and}\,\,\,(I + J_{st}J)\,\,\,\mbox{commute},
\end{eqnarray}
\begin{eqnarray}
\label{com2}
(I + J_{st}J)J = -J_{st}(I+J_{st}J),
\end{eqnarray}
\begin{eqnarray}
\label{com3}
(I-J_{st}J) = J_{st}(I-J_{st}J).
\end{eqnarray}
Then (\ref{com3}) implies
\begin{eqnarray}
\label{com4}
J(I - J_{st}J)^{-1} = (I - J_{st} J)^{-1}J_{st}.
\end{eqnarray}
We show $LJ_{st} = -J_{st}L$.
Using successively (\ref{com1}), (\ref{com4}), (\ref{com2}),
and (\ref{com1}) we obtain
\begin{eqnarray*}
& &LJ_{st} = (I - J_{st}J)^{-1}(I + J_{st}J)J_{st} = (I + J_{st}J)(I-J_{st}J)^{-1}J_{st} =\\
& &(I + J_{st}J)J(I-J_{st}J)^{-1} = -J_{st}(I + J_{st}J)
(I-J_{st}J)^{-1} = -J_{st}L.
\end{eqnarray*}
\end{proof} $\blacksquare$
\smallskip

Thus, if we view $(\H,J_{st})$ as a complex vector space, the action of $L$ can be expressed in the form
$$ L h = A_J \bar{h}$$
where $A_J: \H \to \H$ is a bounded $J_{st}$-linear operator. We call $A_J$ the {\it complex representation} of $J$ and often omit $J$.  With this convention the Cauchy-Riemann equations (\ref{CRglobal}) for a $J$-complex disc $Z:\D\to\H$, $Z: \D \ni \zeta \mapsto Z(\zeta)$ can be  written  in  the form
\begin{eqnarray}
\label{holomorphy}
Z_{\bar\zeta}=A_J(Z)\bar Z_{\bar\zeta},\quad
\zeta\in\D.
\end{eqnarray}

In the present paper an almost complex structure $J$ will arise as the direct image $J = \Phi_*(J_{st}):= d\Phi \circ J_0 \circ d\Phi^{-1}$ of $J_{st}$ under a symplectomorpfism $\Phi: (\H,\omega) \to (\H,\omega)$. We discuss assumptions on $\Phi$ that allow to deduce the equations (\ref{holomorphy}).

Let $\Phi:G' \to G$ be a symplectomorphism of class $C^1$ between  open bounded subsets $G'$ and $G$ of $\H$.  Suppose  that  the tangent maps $d\Phi$ {\it are  uniformly bounded on $G'$}. Then by Proposition  \ref{PropAC3} (a), the tangent maps $d\Phi^{-1}$ also are uniformly bounded. Set $J = \Phi_*(J_{st})$.

Put $P(Z) = \Phi_Z(Z)$ and $Q = \Phi_{\bar Z}(Z)$.  Since the operator $P$ is invertible by Proposition \ref{PropAC3} (b), it follows from \cite{SuTu2} that
\begin{eqnarray}
\label{A}
A_J = Q \bar P\,^{-1}.
\end{eqnarray}
Indeed, the proof of Lemma 2.3 from \cite{SuTu2} can be carried to the Hilbert case without changes and gives (\ref{A}). Hence  Proposition \ref{PropAC3} (c) implies that there exists a constant $0 < a < 1$ such that
\begin{eqnarray}
\label{norm2}
\| A_J(Z) \| \leq a < 1
\end{eqnarray}
for all $Z \in G$.

\begin{lemma}
\label{LemInverI}
Let $B: X \to X$ be a linear operator on a real Hilbert space $X$. Suppose $\langle Bx, x \rangle > 0$ for all $x \neq 0$.
Then $I + B$ is invertible and $\| Lx \| < \| x \|$, $x \neq 0$, here $L = (I + B)^{-1}(I-B)$. Conversely, if
$I + B$ is invertible and $\| L x \| < \| x \|$ for all $x \neq 0$, then $\langle Bx,x\rangle > 0$ for all $x \neq 0$.
\end{lemma}
\begin{proof} We can consider $B$ as a complex linear operator on $X \otimes_\R \C = X^\C$. Since $\langle Bx,x\rangle > 0$, $x \in X$, $x \neq 0$, for every spectral value $\lambda \in \sigma(B)$, we have
$\Re \lambda \ge 0$. Hence $0\notin\sigma(I + B)$, and $I + B$ is invertible.

Put $y = (I + B)^{-1}(I - B)x$, $x \neq 0$. Then $x-y = B(x+y)$. Put $u = x+y$, $v = x-y= Bu$. By the hypothesis, $\langle u,v\rangle  > 0$. Since $\| u \pm v \|^2 = \| u \|^2 \pm 2 \langle u,v \rangle + \| v \|^2$,
we have $\| u + v \| > \| u - v \|$, that is, $\| y \| < \| x \|$. Hence $\| Lx \| < \| x \|$.
The converse is obtained along the same lines.
\end{proof} $\blacksquare$

\begin{prop}
\label{PropAC4}
Let $J$ be a linear almost complex structure tamed by the standard symplectic form $\omega$. Then $J_{st} + J$ is invertible, and  $\| Lx \| < \| x \|$ for $x \neq 0$; here $L = (J_{st} + J)^{-1}(J_{st} - J)$.
\end{prop}
\begin{proof} As a bilinear form,
$$
\omega(x,y) = \frac{i}{2} \left(\langle x, y \rangle - \bar{\langle x,y \rangle}\right).
$$
Since $J$ is tamed, for $x \neq 0$,
$$
\omega(x,Jx) = \frac{i}{2} \left(\langle x,Jx \rangle - \bar{\langle x,Jx \rangle} \right) > 0.
$$
Then $\Re\langle x,-J_{st}Jx\rangle > 0$. Put $B = -J_{st}J$. By Lemma \ref{LemInverI}, $I + B$ is invertible. Then
$$L = (J_{st} + J)^{-1}(J_{st} -J) = (I + B)^{-1}(I-B).$$
By Lemma \ref{LemInverI}, $| Lx | < | x |$, $x \neq 0$.
\end{proof} $\blacksquare$
\smallskip

Since $L$ is $\C$-antilinear, it follows that
$ L = \{ 0, A \}$, $\| A \| \le 1$.

\begin{prop}
\label{PropAC7}
Let $J$ be a linear almost complex structure on $\H$ tamed by $\omega$. Then $J$ is compatible with $\omega$ if and only if
$B^* = B$ or equivalently $A^t = A$; here $B = -J_{st}J$ and
$$L = (J_{st} + J)^{-1} (J_{st} - J) = (I + B)^{-1}(I-B) = \{ 0, A \}.$$
Moreover in this case $\| L \| = \| A \| < 1$.
\end{prop}
\begin{proof}
Suppose $J$ is compatible with $\omega$, that is, $J$ is a linear symplectomorphism. Let $J =\{ P, Q \}$. By Proposition \ref{PropAC3}, $J^{-1} = \{ P^*,-Q^t \}$. Since $J^{-1} = -J$, we have $P^* = -P$ and $Q^t = Q$. The latter imply that
$B = \{-iP, -iQ \}$ is self-adjoint. Also $L^*=L$, hence $A^t = A$. Since $B$ is invertible, $0 \notin\sigma(B)$. Since $J$ is tamed, $B \ge 0$. Hence $\sigma(B)$ is contained in $[\lambda_0,\| B \| ]$, here $\lambda_0 > 0$ is the minimal spectral value of $B$.

By spectral mapping theorem,
$$\sigma(L) =\{ (1-\lambda)(1+\lambda)^{-1}: \lambda \in \sigma(B) \}.$$
Since $L$ is self-adjoint,
$$\| L \|  =\max \{ |(1-\lambda)(1+\lambda)^{-1}| : \lambda \in \sigma(B) \}.$$
Since the function $\lambda \mapsto (1-\lambda)(1+\lambda)^{-1} > -1$ is decreasing,
$$\| L \| = \max \{ (1-\lambda_0)(1+\lambda_0)^{-1},  (1-\| B \|)(1+\| B \|)^{-1} \} < 1.$$
The rest of the conclusions are obvious.
\end{proof} $\blacksquare$

In the finite-dimensional case Proposition \ref{PropAC7} also holds for tamed almost complex structures. The following example shows that in general this is not true in the Hilbert case.
\smallskip

{\bf Example.} We construct a tamed linear almost complex structure $J$ on $\H$ for which $\| A \| = 1$.
Let $J = \{ P,Q \}$. Put $P =iI$. Then $J^2 = -I$ reduces to $Q \bar{Q} = 0$. Put $B = -J_{st}J = \{ I, -iQ\}$. Then $I + B = \{ 2I, -iQ \}$, $I - B = \{ 0,iQ \}$ and $(I + B)^{-1} = \{\frac{1}{2}I,\frac{1}{4}Q \}$.
Thus  $L = (I + B)^{-1}(I-B) = \{ 0,\frac{i}{2}Q \}$.
Hence $A=\frac{i}{2}Q$ and $\| A \| = \frac{1}{2} \| Q \|$. The structure $J$ is tamed if and only if  $\Re \langle Bz,z \rangle > 0$ for $z \neq 0$. This condition reduces to
\begin{eqnarray}
\label{*}
\Re \langle Q \bar{z}, z \rangle < | z |^2,\,\,\, z \neq 0.
\end{eqnarray}
We now construct $Q$ satisfying (\ref{*})  with $\| Q \| = 2$, hence $\| A \| = 1$. We represent $\H = \H_1 \oplus \H_2$, the sum of two copies of the Hilbert space.
Define
$$
Q = \left(
\begin{array}{cll}
0 & & 2Q_0\\
0 & & 0
\end{array}
\right),
$$
here $Q_0$ is the diagonal operator
$Q_0=\text{Diag}(c_1, c_2, \ldots)$, $(c_n)$ is a real sequence, $0 < c_n < 1$, $c_n \to 1$ as $n \to \infty$. Clearly $\| Q \| = 2$, $Q \bar{Q} = 0$. For
$z = z_1 + z_2$, $z_j \in \H_j$, we have $\Re \langle Q\bar{z},z \rangle = 2\Re \langle Q_0 z_1,\bar{z}_2 \rangle$. If $z \neq 0$, then clearly (\ref{*}) is fulfilled because $0 < c_n < 1$. $\blacksquare$

\section{Some properties of the Cauchy integral}

The main analytic tool in the theory of pseudoholomorphic curves is the Cauchy integral.
In this section we recall some important regularity properties of the Cauchy (Cauchy-Green) integral and related integral operators and generalize them to Hilbert space-valued functions. They are crucial for our method because we employ them in order to solve boundary value problems for Beltrami type equation.

Everywhere $\zeta$, $z$ and $t$  denote scalar complex variables. Denote by $\D = \{ \zeta \in \C: | \zeta | < 1 \}$ the unit disc in $\C$.

\subsection{Modified Cauchy integrals}

Let  $f:\D \to \C$ be a measurable function. The {\it Cauchy (Cauchy-Green) operator} is defined by
\begin{eqnarray}
\label{Cauchy1}
Tf(z) = \frac{1}{2\pi i}\int_\D \frac{f(t) dt \wedge d\bar{t}}{t-z}.
\end{eqnarray}
The {\it Beurling integral operator} is the formal derivative of $T$, i.e.,
\begin{eqnarray}
\label{Cauchy2}
Sf(z) = p.v. \frac{1}{2\pi i}\int_\D  \frac{f(t) dt \wedge d\bar{t}}{(t-z)^2}
\end{eqnarray}
It is classical  that $T: L^p(\D) \to W^{1,p}(\D)$ is bounded for $p > 1$ and $(\partial/\partial\bar{\zeta}) Tf = f$ as Sobolev's derivative, i.e., $T$ solves the $\bar\partial$-problem in $\D$. Furthermore, $Tf$ is holomorphic on $\C \setminus \bar{\D}$. There are additional properties.

\begin{prop}
\label{operators}
Set $g = Tf$ and $h = S f$.  Then the following holds.
\begin{itemize}
\item[(i)] If  $f \in L^p(\D)$ , $p > 2$ then $g \in C^\alpha(\bar \C) \cap L^\infty(\C)$ with $\alpha = (p-2)/p$. More precisely,
there exist constants $C_1 = C_1(p)$ and $C_2 = C_2(p)$ such that
\begin{eqnarray*}
& &| g(\zeta) | \leq C_1 \| f \|_{L^p(\D)},\\
& &| g(\zeta_1) - g(\zeta_2) | \leq C_2 \| f \|_{L^p(\D)} | \zeta_1 - \zeta_2 |^\alpha
\end{eqnarray*}
for every $\zeta,\zeta_1, \zeta_2 \in \C$.
\item[(ii)]  Let $f \in C^{m,\alpha}(\D)$, for an integer $m \geq 0$ and $0 < \alpha < 0$. Then $g \in C^{m+1,\alpha}(\D)$ and
$T:C^{m,\alpha}(\D) \to C^{m+1,\alpha}(\D)$ is a bounded linear operator. Furthermore $\partial_{\zeta} g = h$. The linear operator $S: C^{m,\alpha}(\D) \to C^{m,\alpha}(\D)$ is bounded.
\item[(iii)] The operator $S$ can be uniquely extended to a bounded linear operator $S: L^p(\D) \to L^p(\D)$
for any $p > 1$. If $f \in L^p(\D)$, $p > 1$ then $(\partial/\partial \zeta) g = h$ as a Sobolev derivative.
\end{itemize}
\end{prop}

The proofs are contained in \cite{AIM,Ve}. We introduce modifications of the above integral operators useful for applications to boundary value problems.

Consider distinct complex numbers $z_k$, $k=1,...,n$ , $| z_k | = 1$ and real $0 < \alpha_k < 1$, $k=1,...,n$.  Let
$$
Q(z) = \prod_{k=1}^n (z - z_k)^{\alpha_k}.
$$
Here we make the cuts $\Gamma_k = \{ \lambda z_k: \lambda > 0  \}$ and fix a branch of $Q$ on $\D \cup (\C \setminus \cup_k \Gamma_k)$.
Define
\begin{align*}
T_Qf(z) &= Q(z) \left ( T(f/Q)(z) + z^{-1} \bar{T(f/Q)(\bar{z}^{-1})} \right ) \\
&= Q(z) \left ( \frac{1}{2\pi i}\int_\D \frac{f(t)dt\wedge d\bar{t}}{Q(t)(t-z)} + \frac{1}{2\pi i}\int_\D \frac{\bar{f(t)}dt \wedge d\bar{t}}{\bar{Q(t)}(\bar{t} z -1)} \right ).
\end{align*}
In order to simplify notations, we write
$$\partial = \frac{\partial}{\partial z}, \,\,\,\, \bar{\partial} = \frac{\partial}{\partial \bar z}.$$
Define the operator
\begin{eqnarray*}
S_Q = \partial T_Q f
\end{eqnarray*}
as the weak derivative of $T_Qf$. The following result is contained in \cite{Mo}.

\begin{thm}
\label{TheoCauchy1}
Let $p_1 < p < p_2$, where
$$1 < p_1 = \max_k \frac{2}{2-\alpha_k} < 2 < p_2 = \min_k \frac{2}{1 - \alpha_k}.$$
Then
$S_Q: L^p(\D) \to L^p(\D)$ and
$T_Q: L^p(\D) \to W^{1,p}(\D)$
are bounded linear operators.
\end{thm}
We present the proof in Appendix II slightly improving the original argument of \cite{Mo}. We closely follow \cite{Mo} and do not claim originality.

As examples we consider two operators corresponding to two special weights $Q$. Consider the arcs $\gamma_1 = \{ e^{i\theta} : 0 < \theta < \pi/2 \}$, $\gamma_2 = \{ e^{i\theta} : \pi/2 < \theta < \pi \}$, $\gamma_3 = \{ e^{i\theta} : \pi < \theta < 2\pi\}$ on the unit circle in $\C$. Introduce the functions
$$
R(\zeta) = e^{3\pi i/4}(\zeta - 1)^{1/4} (\zeta + 1)^{1/4}(\zeta - i)^{1/2}\quad
\text{and}\quad
X(\zeta)= R(\zeta)/\sqrt{\zeta}.
$$
Here we choose the branch of $R$ continuous in $\bar\D$
satisfying $R(0) = e^{3\pi i/4}$.
For definiteness, we also choose the branch of $\sqrt{\zeta}$ continuous in $\C$ with deleted positive real line, $\sqrt{-1}=i$. Then $\arg X$   on  arcs $\gamma_j$, $j=1,2,3$ is equal to  $3\pi/4$, $\pi/4$ and $0$ respectively.  Therefore, the function $X$ satisfies the boundary conditions
\begin{equation}
\label{BC}
\begin{cases}
\; \Im (1+i)X(\zeta) = 0, & \zeta \in \gamma_1,\\
\; \Im (1-i)X(\zeta) = 0, & \zeta \in \gamma_2,\\
\; \Im X(\zeta) = 0, &      \zeta \in \gamma_3,
\end{cases}
\end{equation}
which represent the lines through 0 parallel to the sides of
the triangle $\Delta$ with vertices at $\pm1$, $i$.
Consider the operators
\begin{eqnarray}
\label{Cauchy3bis}
T_1 = T_Q + 2i\,\Im Tf(1) \,\,\, \mbox{ with} \,\,\,Q = \zeta - 1
\end{eqnarray} and
\begin{eqnarray}
\label{Cauchy4bis}
T_2 = T_Q \,\,\, \mbox{ with} \,\,\,Q = R.
\end{eqnarray}
Note that
$$
T_1 f(\zeta) = Tf(\zeta) - \bar{Tf(1/\bar{\zeta})}.
$$
The formal derivatives of these operators  are denoted by
\begin{eqnarray}
\label{Cauchy5}
S_jf(\zeta) = \frac{\partial}{\partial\zeta} T_jf(\zeta)
\end{eqnarray} as integrals in the sense of the Cauchy principal value.
As a consequence of
the above results, we have

\begin{prop}
\label{OpBounValScal}
The operators $T_j,S_j$ enjoy the following properties:
\begin{itemize}
\item[(i)] Each $S_j :L^p(\D) \to L^p(\D)$, $j=1,2$, is a bounded linear operator for $p_1 < p < p_2$. Here for $S_1$ one has $p_1 = 1$ and $p_2 = \infty$, and for $S_2$ one has $p_1 = 4/3$ and $p_2 = 8/3$. For $p_1<p<p_2$, one has $S_jf(\zeta) = (\partial/\partial\zeta) T_jf(\zeta)$ as Sobolev's derivatives.
\item[(ii)]  Each $T_j :L^p(\D) \to W^{1,p}(\D)$, $j=1,2$,  is a bounded linear operator for $p_1\le p<p_2$. For $f \in L^p(\D)$, $p_1<p<p_2$, one has $(\partial/\partial\bar{\zeta}) T_j f = f$ on $\D$ as Sobolev's derivative.
\item[(iii)]  For every $f \in L^p(\D)$, $2<p<p_2$, the function $T_1f$ satisfies $\Re T_1f|_{b\D} = 0$
    whereas $T_2f$ satisfies
    the same boundary conditions (\ref{BC}) as $X$.
\item[(iv)] Each $S_j: L^2(\D) \to L^2(\D)$, $j=1,2$, is an isometry.
\item[(v)]  The function $p \mapsto \| S_j \|_{L^p}$ approaches $\| S_j \|_{L^2} = 1$ as $p\searrow 2$.
\end{itemize}
\end{prop}

Our next goal is to extend the previous results on the Cauchy integral in Sobolev classes to Hilbert space-valued functions.

\subsection{Bochner's integral}

Following \cite{Y} we recall basic properties of Bochner's integral.
Let $(S,\mu)$ be a measure space, $X$ be a  Banach space $X$ and  $X'$ be the dual of $X$. A map $u:S \to X$ is called {\it weakly measurable} if, for any $f \in X'$, the function $S \ni s\mapsto f(u(s))$ is measurable. A map $u$ is called {\it simple} or {\it finitely-valued} if it is constant $\neq 0$ on each of a finite number disjoint measurable sets $B_j$ with $\mu(B_j) < \infty$ and $u = 0$ on $S \setminus \cup_j B_j$. A map $u$ is called {\it strongly measurable} if there exists a sequence of simple functions strongly convergent to $u$ a.e. on $S$.  Suppose that $X$ is separable. Then $u$ is strongly measurable if and only if it is weakly measurable. This fact is a special case of Pettis's theorem. We will deal with the case where $X$ is a separable Hilbert space, so  these two notions of measurability will coincide.

Consider a simple function  $u:S \to X$; let $u = x_j$ on $B_j$, $j=1,...,n$, where $B_j$'s are disjoint  and $\mu(B_j) < \infty$ and $u = 0$ on $S \setminus \cup_j B_j$.
Then we put
$$
\int_S u(s) d\mu(s) = \sum_{j=1}^n x_j\mu(B_j).
$$
A function $u:S \to X$ is called {\it Bochner integrable} if there exists a sequence $(u_k)$ of simple functions strongly convergent to $u$ a.e. on $S$ such that
$$
\lim_{k \rightarrow \infty} \int_S \| u(s) - u_k(s) \| d\mu(s) = 0.
$$
Then {\it the Bochner integral of $u$} is defined by
$$
\int_S u(s) d\mu(s) = \lim_{k\rightarrow \infty} \int_S u_k(s)d\mu(s),
$$
where the limit in the right hand denotes the strong convergence. One can show that this definition is consistent i.e. is independent of the choice of the sequence $(u_k)$. The fundamental theorem of Bochner states that a strongly measurable function $u$ is Bochner integrable if and only if the function $s \mapsto \| u(s) \|$ is integrable. Furthermore, Bochner's integral enjoys the following properties:
\begin{itemize}
\item[(i)] One has
\begin{eqnarray}
\label{Bochner1}
\left\| \int_S u(s) d\mu(s)\right\| \le \int_S \| u(s) \| d\mu(s).
\end{eqnarray}
\item[(ii)] Let $L:X \to Y$ be a bounded linear operator between two Banach spaces. Assume that  $u:S \to X$ is a Bochner integrable function. Then $L u$ is a Bochner integrable function, and
\begin{eqnarray}
\label{Bochner2}
\int_S L u(s) d\mu(s) = L \int_S u(s) d\mu(s).
\end{eqnarray}
\end{itemize}

In our applications we deal with the case where $S = \D$ or another subset of $\C$ and $X =\H$ is a Hilbert space. Denote by   $W^{k,p}(\D,\H)$  the  Sobolev classes of maps $Z:\D \to \H$ admitting the $p$-integrable weak partial derivatives $D^\alpha Z$ up  to the order $k$ (as usual we identify functions coinciding almost everywhere).  We define weak derivatives  in the usual way using the space of scalar-valued test functions. We write simply $L^p$ if $k=0$.
The norm on $L^{p}(\D,\H)$ is defined by
$$
\| Z \|_{L^p(\D,\H)} = \left ( \int_{\D} \| Z(\zeta) \|_\H^p (i/2)d\zeta \wedge d\bar{\zeta}\right )^{1/p}.
$$
The space $W^{k,p}(\D,\H)$ equipped with the norm
$$\| Z \| = \left ( \sum_{| \alpha | \leq k} \| D^\alpha Z \|^p_{L^p(\D,\H)}\right )^{1/p}
$$
is a Banach space.

We define Lipshitz spaces $C^{k,\alpha}(\D,\H)$, $0 < \alpha \le 1$, $k$ is a positive integer,  in the usual way.
If $\H = \C$, we as usual  write $L^p(\D)$, $W^{1,p}(\D)$ and $C^{k,\alpha}(\D)$ respectively.
We note that the system (\ref{holomorphy}) still makes sense for
$Z \in W^{1,p}(\D)$ for $p \ge 2$.

\subsection{Linear operators in vector-valued Sobolev spaces and their extension}

For definiteness  we only consider the functions $\D \to \H$, where as usual $\H$ is a separable Hilbert space.

Let $P: L^p(\D) \to L^p(\D)$ be a bounded linear operator.We say that {\it $P$ extends to $L^p(\D,\H)$} if there is a unique
bounded linear operator $P_\H: L^p(\D,\H) \to L^p(\D,\H)$ such that for every $u \in L^p(\D)$ and $h \in \H$ we have
$P_\H(uh) = P(u)h$. We will usually omit the index $\H$ in $P_\H$. The next proposition concerns the properties of the integral operators $T$, $T_1$, $T_2$ introduced in Subsection 3.1.

\begin{prop}
\label{VectIntPropI}
\begin{itemize}
\item[(i)] Every bounded linear operator $P:L^p(\D) \to L^p(\D)$ extends to $L^p(\D,\H)$, $1 \le p < \infty$.
\item[(ii)] For $p > 2$ the operators $T$, $T_1$ are bounded linear operators $L^p(\D,\H) \to C^\alpha(\D,\H)$ with
$\alpha = (p-2)/p$.
\item[(iii)] For $u \in L^p(\D,\H)$ for appropriate $p$ as in Proposition \ref{OpBounValScal}, we have
\begin{eqnarray*}
\frac{\partial Tu}{\partial \bar \zeta} = u, \,\,\frac{\partial T_ju}{\partial \bar \zeta} = u, \,\,
\frac{\partial Tu}{\partial  \zeta} = Su, \,\,\frac{\partial T_ju}{\partial  \zeta} = S_ju , \,\, j=1,2
\end{eqnarray*}
as weak derivatives.
\item[(iv)] The operators $T$, $T_1$, $T_2$ are bounded linear operators $L^p(\D,\H) \to W^{1,p}(\D,\H)$ for the same $p$ as in Proposition  \ref{OpBounValScal}.
\end{itemize}
\end{prop}
\begin{proof} (i) If $P$ is a singular integral operator, the result follows because $\H$ is a UMD space \cite{B}. For a general bounded linear operator the result follows because $\H$ is so called $p$-space \cite{He}, which means exactly the same as Proposition \ref{VectIntPropI} (i). Since the operators $T$, $T_1$, $T_2$, $S$, $S_1$, $S_2$ are bounded linear operators in $L^p(\D)$ for appropriate $p > 1$, they extend to $L^p(\D,\H)$. Note that these extended operators preserve the same norms. The parts (ii), (iii), and (iv) are proved in \cite{SuTu4}
\end{proof} $\blacksquare$
\medskip

{\bf Remark.} Let $\{ e_n \}_{n=1}^\infty$ be an orthonormal basis of $\H$. Then every $u \in L^p(\D,\H)$, $p \ge 1$, is represented by the series
\begin{eqnarray}
\label{series}
u = \sum_{n=1}^\infty u_n e_n
\end{eqnarray}
converging in $\H$ a.e. in $\D$. Here $u_n(\zeta) = \langle u(\zeta),e_n \rangle$ is measurable, hence $u_n \in L^p(\D)$, $\| u_n \|_p \le \| u \|_p$. It is easy to see that for every $u \in L^p(\D,\H)$, $p \ge 1$, the series (\ref{series}) converges a.e. in $\D$ if and only if it converges in $L^p(\D,\H)$. Furthermore, if $P:L^p(\D) \to L^p(\D)$, $p \ge 1$, is a bounded linear operator and  $u \in L^p(\D,\H)$ is given by (\ref{series}),  then
$$
P_\H u = \sum_{n=1}^\infty (Pu_n) e_n.
$$

\subsection{Cauchy integral for Lipschitz classes of vector functions}

Above we considered the properties of the Cauchy integral for Sobolev classes of vector functions. The Lipschitz classes also are useful for applications. Here the situation is simpler and the proofs follow the scalar case line-by-line with obvious changes (essentially the module must be replaced by the Hilbert space norm). For this reason we omit proofs.

The {\it Cauchy type integral} of a function $f:b\D \to \H$
\begin{eqnarray}
\label{VectorCauchy1}
Kf(z)=\frac{1}{2\pi i}\int_{b\D} \frac{f(\zeta)}{\zeta - z}d\zeta
\end{eqnarray}
is defined for $z \in \C \setminus b\D$. Similarly to the scalar case, the Cauchy type integral is holomorphic on $\C \setminus b\D$.

\begin{thm}
\label{VectorCauchyTm}
We have:
\begin{itemize}
\item[(i)] Let $f \in C^{m,\alpha}(b\D,\H)$, $0 < \alpha < 1$, $m\ge 0$ is an integer.
 Then  $Kf \in C^{m,\alpha}(\D,\H)$ and $K: C^{m,\alpha}(b\D,\H) \to C^{m,\alpha}(\D,\H)$ is a bounded linear operator.
\item[(ii)] Let $f \in C^{m,\alpha}(\D)$, $0 < \alpha < 1$, $m\ge 0$ be an integer.
Then the Cauchy-Green integral $Tf$ is of class $C^{m+1,\alpha}(\D,\H)$ and $T:C^{m,\alpha}(\D,\H) \to C^{m+1,\alpha}(\D,\H)$ is a bounded linear operator.
\end{itemize}
\end{thm}

For the proof of (i) in the scalar case see \cite{Ga} when $m=0$ and \cite{Ve} for $m \ge 1$.
Let  $f \in C^{0,\alpha}(b\D,\H)$.  The classical  argument
deals with the integrals  of the form
\begin{eqnarray*}
\frac{1}{2\pi i} \int_{b\D} \frac{f(\zeta) - f(\zeta_0)}
{\zeta - z}d\zeta,
\end{eqnarray*}
where $\zeta_0 \in b\D$. The property (\ref{Bochner1}) of Bochner's integral shows that the estimates of these integrals performed in \cite{Ga} for scalar functions, literally go through for vector functions.  This allows us to establish  the Plemelj-Sokhotski formulae for $Kf$ and to  deduce that the boundary values of $Kf$ on $b\D$ satisfy the $\alpha$-Lipschitz condition quite similarly to the scalar case \cite{Ga}. Then $Kf \in C^{0,\alpha}(\D)$  for example, by the classical Hardy-Littlewood theorem (see \cite{Gol}); its proof can be extended to the vector case without changes. This is the only type of modifications which are required in order to extend the proofs of \cite{Ga} and \cite{Ve} from the scalar case to  the case of vector functions.

The proof of (ii) is contained in \cite{Ve} for  scalar functions. This proof is based on properties  of  integrals  of the form
$$
\int_\D \frac{g(\zeta)}{(\zeta - z_1)^\alpha (\zeta - z_2)^\beta}(i/2)d\zeta\wedge d\bar\zeta.
$$
Here $z_j \in \D$, $\alpha,\beta > 0$ and a function $g:\D \to \H$ coincides with  $f(\zeta)$ or with $f(\zeta) \pm f(z_j)$. Note that integrals along the boundary arising in \cite{Ve} disappear in our case since we deal with the circle. Applying the estimate (\ref{Bochner1}), we reduce the estimates of these integrals to the estimates of their scalar kernels performed in \cite{Ve}. The argument of \cite{Ve} literally goes through for the case of vector functions.

\section{Local existence and regularity of pseudoholomorphic discs}

In this section we establish two basic properties of pseudoholomorphic curves: the local existence (Nijenhuis-Woolf's theorem in the finite dimensional case, see \cite{Aud}) and the interior regularity. Since these properties are local, it suffices to establish them for ``small'' discs.

\subsection{Local existence}

Let $\H$ be a Hilbert space (identified with complex $l_2$) and $J$ be an almost complex structure on $\H$. Denote by $\B^\infty = \{ Z \in \H: \| Z \| < 1 \}$ the unit ball in $\H$. A simple but very useful fact is that in a neighborhood $p + r\B^\infty$ of   every point $p \in \H$ the structure $J$ can be represented as a small perturbation (in every $C^k$ norm) of the standard structure $J_{st}$; furthermore, the size of perturbation decreases to $0$ as $r \to 0$.

More precisely, we have the following lemma.

\begin{lemma}
\label{localcordinates1}
For every point $p \in
\H$, every $k \geq 1$  and every $\lambda_0 > 0$ there exist a neighborhood $U$ of $p$ and a
coordinate diffeomorphism $Z: U \to \B^\infty$ such that
\begin{itemize}
\item[(i)] $Z(p) = 0$,
\item[(ii)] $Z_*(J)(0) = J_{st}$,
\item[(iii)]  the direct image $Z_*(J)$ satisfies $|| Z_*(J) - J_{st}
||_{C^k(\bar {\B^\infty})} \leq \lambda_0$.
\end{itemize}
\end{lemma}
\begin{proof} The linear almost complex structure $J(p)$ is equivalent to $J_{st}$. Hence
there exists a diffeomorphism $Z$ of a neighborhood $U'$ of
$p \in \H$ onto $\B^\infty$ satisfying (i) and (ii). Given $\lambda > 0$ consider the dilation
$d_{\lambda}: h \mapsto \lambda^{-1}h$ for  $h \in \H$ and the composition
$Z_{\lambda}: = d_{\lambda} \circ Z$. Consider the direct image $J_\lambda = (Z_\lambda)_*(J_{st})$.
Then $\lim_{\lambda \rightarrow
0} || J_{\lambda} - J_{st} ||_{C^k(\bar
{\B^\infty})} = 0$. Setting $U = Z^{-1}_{\lambda}(\B^\infty)$ for
$\lambda > 0$ small enough, we obtain the desired statement.
\end{proof} $\blacksquare$

The central result of this section is the following

\begin{thm}
\label{NiWo}
Let $(\H,J)$ be a Hilbert space with an  almost complex structure. For integer $k \geq 1$, and $0 < \alpha <1$, every point $p \in \H$ and every tangent vector $v \in T_p\H$ there exists a $J$-holomorphic map $f:\D \to M$ of class $C^{k,\alpha}(\bar\D)$ such that $f(0) = p$ and $df_0 (\partial/\partial \Re\zeta) = t v$ for some $t > 0$.
\end{thm}
\begin{proof} We suppose that local coordinates near $p=0$ are chosen by Lemma \ref{localcordinates1}. Its proof provides us with  the family $(J_\lambda)$ of almost complex structures over the ball $\B^\infty$ smoothly depending on the parameter $\lambda \ge 0$,and $J_0 = J_{st}$.
Each structure $J_\lambda$ is equivalent to the initial structure $J$ in a neighborhood $U_\lambda$ of $p$. A map $Z:\D \to \B^\infty$ is $J_\lambda$-holomorphic if and only if its satisfies the Cauchy-Riemann equations
\begin{eqnarray}
\label{discholomorphicity1}
Z_{\bar\zeta} - A_{\lambda}(z) \bar{ Z}_{\bar\zeta} = 0.
\end{eqnarray}
Here we use the notation $A_{\lambda} = A_{J_\lambda}$ for the complex representation of the structure $J_\lambda$.Note that
$$A_{\lambda}(0) = 0$$
because $J_{\lambda}(0) = J_{st}$. Note that $A_0 \equiv 0$.

Using the Cauchy-Green operator $T$ in $\D$ we replace equation (\ref{discholomorphicity1}) by an integral equation
\begin{eqnarray}
\label{MainIntEq}
Z + T A_{\lambda}(Z) \bar{ Z}_{\bar\zeta} = W,
\end{eqnarray}
where $W \in C^{k,\alpha}(\D,\H)$ is a holomorphic (in the usual sense) vector function in $\D$. Recall that $\bar\partial \circ T = I$. Therefore, given $W$ of this class, a solution $Z$ to (\ref{MainIntEq}) automatically is also a solution to (\ref{discholomorphicity1}). Fix $\lambda_0 > 0$ small enough and denote by $S$ the class of maps $Z \in C^{k,\alpha}(\D,\H)$ such that $Z(\D) \subset \B^\infty$.

Consider the map
$$\Phi: [0,\lambda_0] \times S \to C^{k,\alpha}(\D,\H),$$
$$\Phi: (\lambda,Z) \mapsto Z + TA_{\lambda}(Z)\bar{Z}_{\bar\zeta}.$$

This map is well defined by the regularity of $T$ (Theorem \ref{VectorCauchyTm} (ii)) and is smooth in $(\lambda,Z)$. We view $\lambda$ as a parameter and use the notation $\Phi_\lambda:= \Phi(\lambda,\bullet)$. Note that $\Phi_0 = I$, $\Phi_\lambda(0) = 0$ and $d\Phi_\lambda(0) = I$. By the implicit function theorem there exists the inverse map  $\Psi_\lambda = (\Phi_\lambda)^{-1}$ defined in a neighborhood ${\mathcal U}$ of the  origin in $C^{k,\alpha}(\D)$. The family $\Psi_\lambda$   smoothly depends on $\lambda$ and $\Psi_0 = I$.

Let $r> 0$ be small enough such that $2r\B^\infty \subset {\mathcal U}$. For $q$ and $v$ in $r\B^\infty$
consider the map $W_{q,v}(\zeta) = q + \zeta v$ holomorphic in $\zeta \in \D$. This is the usual complex line through $q$ in the direction $v$. Then $Z_{q,v,\lambda}:= \Psi_\lambda(W_{q,v})$ is a $J_\lambda$-holomorphic disc. Define the evaluation map
$$Ev_\lambda: (q,v) \mapsto (Z_{q,v,\lambda}(0), dZ_{q,v,\lambda}(0)(\partial/\partial\Re\zeta)).$$
Then $Ev_0 = I$. Hence for $\lambda$ sufficiently close to $0$ the map $Ev_\lambda$ is a diffeomorphism between neighborhoods of the origin in $\H \times \H$.
\end{proof} $\blacksquare$

\subsection{Interior regularity} The equation (\ref{holomorphy}) makes sense if $z$ belongs to the Sobolev space $W^{1,p}(\D,\H)$, $p > 2$. However its  ellipticity   implies the regularity of generalized solutions.

\begin{thm}
\label{InnerDiscRegularity}
Suppose that $A : \H \to {\mathcal L}(\H)$ is a map of class $C^{k,\alpha}$ for some integer $k \geq 1$ and $0 < \alpha < 1$. Assume also $A$ is small enough in the $C^{k,\alpha}$ norm. Let $p > 2$ be such that $\alpha < \beta = 1 - 2/p$. Then the solutions of (\ref{holomorphy}) in $W^{1,p}(\D,\H)$ are of class $C^{k,\alpha}(\D,\H)$.
\end{thm}
\begin{proof} Fix $\varepsilon > 0$ and a real cut-off function $\chi \in C^\infty(\C)$ such that $supp (\chi) \in (1- \varepsilon)\D$, and $\chi \equiv 1$ on $(1 - 2\varepsilon)\D$. Set $W = \chi Z$. Then $W$ satisfies the linear non-homogeneous equation
\begin{eqnarray}
\label{DiscReg1}
W_{\bar\zeta} + a \bar W_{\bar\zeta} = b
\end{eqnarray}
with $a(\zeta) = -(A \circ Z)(\zeta)$ and $b(\zeta) = \chi_{\bar\zeta}(a \bar Z + Z)(\zeta)$.
By the assumptions of the theorem and the Morrey-Sobolev embedding (Theorem \ref{Morrey}), these coefficients are of class $C^\alpha(\D,\H)$. The coefficient $b$ extends by $0$ as a $C^\alpha$ function on the whole complex plane $\C$. Furthermore, since
$supp (W) \subset (1- \varepsilon)\D$, we can multiply $a$ by a suitable cut-off function equal to 1 on $(1 - \varepsilon)\D$ and vanishing outside $\D$. Then the equation (\ref{DiscReg1}) does not change. This equation is equivalent to the integral equation
\begin{eqnarray}
\label{DiscReg2}
W + T a  \bar{W}_{\bar\zeta} = Tb + h,
\end{eqnarray}
where $h$ is a usual holomorphic function on $\C$ and $T$ is the Cauchy-Green operator in $\D$. Since $W$, $T a  \bar{W}_{\bar\zeta}$ and $Tb$ are bounded at infinity, we conclude that $h \equiv 0$. The linear operator
$$
L: C^{1,\alpha}(\D,\H) \to
C^{1,\alpha}(\bar\D,\H),\qquad
L: W \mapsto W + T a  \bar{W}_{\bar\zeta}
$$
is well-defined and bounded by Theorem \ref{VectorCauchyTm} (ii). This operator is invertible since the norm of $a$ is small. Hence the equation (\ref{DiscReg2}) admits a unique solution in $C^{1,\alpha}(\D,\H)$. However, the same operator $L$ viewed as $L: W^{1,p}(\D,\H) \to W^{1,p}(\D,\H)$ also has the trivial kernel. Hence, $W$ is of class $C^{1,\alpha}(\D,\H)$. We conclude the proof by iterating this argument.
\end{proof} $\blacksquare$

\section{Boundary value problems for $J$-holomorphic discs}

Let $E$ be a closed real submanifold in a Hilbert space $(\H,J)$. For applications of theory of pseudoholomorphic curves it is important to construct $J$-holomorphic discs $Z:\D \to \H$ with boundary attached to $E$. Of course, here we consider discs which are at least continuous on $\bar{\D}$ (in fact, they usually belong to Sobolev classes $W^{1,p}(\D,\H)$ with $p > 2$, so they are $\alpha$-Lipschitz in $\D$ by the Morrey-Sobolev embedding). As usual, we say that the boundary of such a disc is attached or glued to $E$ if $Z(b\D)$ is contained in $E$. Usually $E$ is defined by a finite or infinite system of equations. To be concrete, consider a smooth map $\rho:\H \to X$ and assume that
$$E = \rho^{-1}(0),$$
here $X$ is an appropriate space of finite or infinite dimension.
Then attaching a $J$-holomorphic disc to $E$ reduces to the following boundary value problem
\begin{equation}
\label{BVP}
\begin{cases}
Z_{\bar\zeta}= A_J(Z)\bar Z_{\bar\zeta}, & \zeta \in \D\\
\rho(Z)|_{b\D} = 0.
\end{cases}
\end{equation}
This boundary value problem for a quasi-linear first order PDE in general has non-linear boundary conditions. Even in the finite-dimensional case the general theory of such problems is not available. Gromov was able to construct solutions in some important special cases, for example, when $E$ is a compact Lagrangian submanifold of $\C^n$. As we mentioned in the introduction, his method is based on the compactness theorems and the deformation theory of pseudoholomorphic curves; a direct attempt to extend these techniques to the  Hilbert space case leads to difficulties.
Our approach to this boundary value problem is inspired by the theory of scalar Beltrami equation and allows us to solve (\ref{BVP}) for some special choices of $\rho$ arising in applications. It can be described as follows.

{\bf Step 1.} We replace the boundary value problem (\ref{BVP}) by a system of (singular) integral equations which can be written in the form:
\begin{eqnarray}
\label{OpEq}
F(Z) = Z.
\end{eqnarray}
For some special choices of $\rho$ this step can be done using the  integral operators $T_Q$, $T_j$ and $S_j$ studied in Section 3. They are modifications  of the Cauchy integral so $\bar\partial T_Q = I$. This allows us to use them in order to construct the solutions of (\ref{holomorphy}) as we did in the previous section. Their boundary properties are determined by the choice of their kernels and imply that a solution to (\ref{OpEq}) automatically satisfies the boundary condition from (\ref{BVP}).

{\bf Step 2.} We prove that in suitably chosen spaces of maps $\D \to \H$ the operator $F$ is compact and takes some convex subset (in fact, some ball) to itself. This gives the existence of solution to (\ref{OpEq}) by Schauder's fixed point theorem. In order to obtain the required properties of $F$, the results of Section 3 are crucially used. First, regularity properties of integral operators are necessary in order to define $F$ correctly in suitable functional spaces. The second key information is a precise control over the norms of these integral operators in the spaces $W^{1,p}(\D,\H)$. Note that the compactness of $F$ also requires some regularity in scales of Hilbert spaces.

We illustrate this approach in two special cases important for applications.

\subsection{Gluing discs to a cylinder}

Let $\H$ be a complex Hilbert space with fixed basis.
Let $(\theta_n)_{n=1}^\infty$ be a sequence of positive numbers
such that $\theta_n \to \infty$ as $n \to \infty$,
for example $\theta_n=n$.
Introduce a diagonal operator
$D={{\rm Diag}}(\theta_1, \theta_2, \ldots).$
For $s\in\R$ we define $\H_s$ as a Hilbert space with the following
inner product and norm:
\[
\<x,y\>_s=\<D^s x,D^s y\>, \quad
\|x\|_s=\|D^s x\|.
\]
Thus $\H_0=\H$,
$\H_s=\{x\in \H: \|x\|_s<\infty \}$ for $s>0$, and
$\H_s$ is the completion of $\H$ in the above norm for $s<0$.
The family $(\H_s)$ is called a Hilbert scale corresponding
to the sequence $(\theta_n)$. For $s > r$, the space $\H_s$
is dense in $\H_r$, and the inclusion $\H_s \subset \H_r$ is compact.
We refer to \cite{Ku2} for a detailed account concerning Hilbert scales  and their applications to Hamiltonian PDEs.

We also have the following analog of Sobolev's compactness theorem: the inclusion
\begin{eqnarray}
\label{Sobolev}
W^{1,p}(\D,H_r) \subset C(\D,\H_s),\,\,\, s <r,\,\,\, p>2
\end{eqnarray}
is compact.
This result is well-known \cite{Aub} in the case of vector functions defined on an interval of $\R$. In the case of the unit disc the required result can be deduced from Morrey's embedding: there exists a bounded inclusion $W^{1,p}(\D,\H_r) \to C^{\alpha}(\D,\H_r)$ with $p > 2$, $\alpha = (p-2)/p$ (see Appendix I). By the Arzela-Ascoli theorem the embedding $C^{\alpha}(\D,\H_r) \to C(\D,\H_s)$ is compact, hence (\ref{Sobolev}) is compact.

We now use the notation
\begin{eqnarray}
\label{coordinates3}
Z = (z,w) = (z,w_1,w_2,...)
\end{eqnarray}
for the coordinates in $\H$. Here $z=\langle Z,e_1 \rangle \in\C$. For a domain $\Omega \subset \C$ we define the cylinder $\Sigma_\Omega= \{ Z \in \H: z \in \Omega \}$ in $\H$. Denote by $\Delta$ the triangle $\Delta = \{ z \in \C: 0 < \Im z < 1 - | \Re z | \}$. Note that $\Area(\Delta) = 1$. Put $\Sigma:= \Sigma_\Delta$.

\begin{thm}
\label{ThDiscs0}
Let $A(Z):\H_0 \to \H_0$, $Z \in \H_0$  be a  continuous family of  linear  operators such that $A(Z): \H_s \to \H_s$  is bounded for $s\in [0,s_0]$, $s_0 > 0$ and  $A(Z) = 0$ for  $Z \in\H\setminus\Sigma$. Suppose that
\begin{eqnarray}
\label{norms3}
\| A(Z) \|_{H_s} \le a
\end{eqnarray} for some $a < 1$ and all $Z$.
Then there exists $p > 2$ such that for every point $(z^0,w^0) \in \Sigma$ there is a solution $Z\in W^{1,p}(\D,\H_0)$ of (\ref{holomorphy}) such that $Z(\bar\D)\subset\bar\Sigma$, $(z^0,w^0) \in Z(\D)$, $\Area(Z) = 1$, and
\begin{eqnarray}
\label{BC0}
Z(b\D) \subset
b\Sigma.
\end{eqnarray}
\end{thm}

The proof is given in \cite{SuTu4} (in the finite dimensional case the present method was introduced in \cite{SuTu1}). It follows the general method described above and reduces the problem to solution of an operator equation of type (\ref{OpEq}). Let us present the key idea.

Consider the biholomorphism $\Phi:\D \to \Delta$ satisfying $\Phi(\pm 1) = \pm 1$ and $\Phi(i) = i$. Note that $\Phi \in W^{1,p}(\D)$ for $p\ge 2$ close enough to $2$ by the classical results on boundary behavior of conformal maps. We use the integral operators $T_1$, $T_2$, $S_1$, $S_2$ introduced in Section 3.1.

We look for a solution $Z = (z,w):\D \to \H_0$  of (\ref{holomorphy}) of class $W^{1,p}(\D,\H_0)$, $p > 2$, in the form
\begin{equation}
\label{MR}
\begin{cases}
z = T_2u + \Phi,\\
w = T_1v - T_1v(\tau) + w^0.
\end{cases}
\end{equation}
for some $\tau \in \D$; hence, $w(\tau) = w^0$.
The Cauchy-Riemann equation (\ref{holomorphy}) for $Z$
of the form (\ref{MR}) turns into the integral equation
\begin{eqnarray}
\label{mainsystem22}
\left(
\begin{array}{cl}
 u\\
 v
\end{array}
\right) = A(z,w)\left(
\begin{array}{cl}
 \bar{S_2u} +\bar{\Phi'}\\
 \bar{S_1v}
\end{array}
\right).
\end{eqnarray}
We have to show  that there exists a solution of (\ref{MR}, \ref{mainsystem22}) so that $z(\tau)=z^0$ for some $\tau\in\D$. It follows from the estimates of the norms of operators $S_j$ in Section 3, that given $(z,w) \in C(\D,\H_0)$ the operator in (\ref{mainsystem22}) is a contraction and admits a unique fixed point $(u,v)$ in $L^p(\D,\H_s)$ for $p > 2$ close to $2$ and $s \in [0,s_0]$. With this $(u,v)$, (\ref{MR}) can be viewed as a non-linear equation for $(z,w)$.
Adding to (\ref{MR}) an auxiliary equation explicitly containing $\tau$ as a scalar unknown, we obtain a system of type (\ref{OpEq}) in $C(\D,\H_0)$. Regularity of $A$ in Hilbert scales imposed by the hypothesis of the theorem is used here in order to assure the compactness of Sobolev's embedding $W^{1,p}(\D,\H_s) \to C(\D,\H_0)$, $p > 2$. This in turn gives the compactness of the operator $F$ from (\ref{OpEq}) required by hypothesis of Schauder's theorem. Again using precise estimates of the norms of operators $T_j$ and $S_j$ in the spaces $W^{1,p}(\D,\H_s)$ established in Section 3, we show that the operator $F$ leaves some ball in $C(\D,\H_0)$ invariant which is sufficient (together with the compactness of $F$) in order to apply Schauder's fixed point theorem. See \cite{SuTu1,SuTu4} for details.

\subsection{Gluing J-holomorphic discs to real tori}

Represent $\H$ as the direct sum $\H = \C_z \oplus \H_w$ where $Z = (z,w_1,w_2,...) = (z, w)$.

Let $J$ be an almost complex structure in $\H$
with complex matrix $A$ of the form
\begin{eqnarray}
\label{Aab}
A = \left(
\begin{array}{cll}
a & & 0\\
b & & 0
\end{array}
\right)
\end{eqnarray}
where $a:\H \to \C_z$ and $b:\H \to \H_w$.
Then the equation (\ref{holomorphy}) means that a map $\D\ni\zeta \mapsto (z(\zeta),w(\zeta))\in\D \times \H_w$
is $J$-holomorphic if and only if it satisfies
the following quasi-linear system:
\begin{equation}
\label{mainsystem}
\begin{cases}
\, z_{\bar\zeta}=a(z,w)\bar z_{\bar\zeta}\\
\, w_{\bar\zeta}=b(z,w)\bar z_{\bar\zeta}.
\end{cases}
\end{equation}
We assume that $|a(z,w)|\le a_0<1$,
which implies the ellipticity of the system.

We are looking for pseudoholomorphic discs $Z = (z,w)$ with boundary glued to the ``torus'' $b\D \times \{ w \in \H_w: \| w \| = r \}$ with $r > 0$. This leads to the boundary value problem for (\ref{Aab}) with non-linear boundary conditions.

Our main result here is the following theorem which can be viewed as a generalization of the Riemann mapping theorem.

\begin{thm}
\label{Riemann}
Suppose that for every $Z \in \H$ the operator $A(Z):\H_s \to \H_s$ is bounded for each $s \in [0,s_0]$ with some $s_0 > 0$.
Let
$a:\bar\D\times\bar{(1+\gamma)\B^\infty}\to \C$,
$\gamma>0$, $b:\bar\D\times\bar{(1+\gamma)\B^\infty}\to \H_w$.  Let $0<\alpha<1$.
Suppose $a(z,w)$ and $b(z,w)$ are $C^\alpha$
in $z$ uniformly in $w$ and $C^{0,1}$ (Lipschitz)
in $w$ uniformly in $z$.
Suppose
\begin{eqnarray*}
|a(z,w)|\le a_0 < 1, \qquad
a(z,0)=0, \qquad
b(z,0)=0.
\end{eqnarray*}
Then there exist $C>0$ and integer $N\ge1$
such that for every integer $n \ge N$, every  $0 < r \leq 1$ and every $V \in \H_w$, $\| V \| = r$ (alternatively, there exist $C>0$ and $0<r_0\le1$
such that for every $n\ge0$ and $0<r<r_0$),
the system (\ref{mainsystem}) has a  solution
$(z,w):\bar\D\to\bar\D\times\bar{(1 + \gamma)\B^\infty}$
of class $W^{1,p}(\D,\H)$, for some $p > 2$,
with the properties:
\begin{itemize}
\item[(i)] $|z(\zeta)|=1$, $\| w(\zeta)\|=r$ for $|\zeta|= 1$;
$z(0) = 0$, $z(1)=1$ and $w(1)= V$;
\item[(ii)] $z: \bar \D \to \bar \D$
is a homeomorphism;
\item[(iii)] $\| w(\zeta)\| \le Cr|\zeta|^n$.
\end{itemize}
\end{thm}

The proof for the case of $\C^2$ (i.e., when $w$ is a complex-valued scalar function) with $a$ and $b$ of class
$C^\infty$ is given in \cite{CoSuTu, SuTu3}.
The proof of the present statement is similar and requires only a few modifications in the spirit of \cite{SuTu4}. We briefly describe it below.

We look for a solution of (\ref{mainsystem}) in the form
$z = \zeta e^u$, $w = r\zeta^n e^v$.
Then the new unknowns $u$ and $v$ satisfy
a similar system but with linear boundary conditions.

We reduce the system of PDEs for $u$ and $v$ to a system of
singular integral equations using suitable modifications of
the Cauchy--Green operator (\ref{Cauchy1}) and the Beurling
operator (\ref{Cauchy2}) as in \cite{CoSuTu}. Of course, here we apply them for vector-valued functions. The method in \cite{CoSuTu}
which is based on the contraction mapping principle
and the Schauder fixed point theorem,
goes through under the present assumptions on $a$ and $b$. The assumption of regularity in scales of Hilbert spaces is used similarly to the previous theorem. Of course, here we use the
extension of these operators to the space of vector-valued functions preserving their norm as described in Section 3. This is crucial for the application of the Schauder fixed point theorem.  All technical work, including regularity properties of the integral operators and their norms estimates, is done in Section 3, so now the method of \cite{CoSuTu,SuTu4} goes through literally.
This gives the existence of solutions $z$, $w$ with the required
properties (i)-(iii) in the Sobolev class
$W^{1,p}(\D,\H)$ for some $p > 2$.

\section{Non-squeezing for symplectic transformations}

Consider the space $\R^{2n}$ with the coordinates $(x_1,...,x_n,y_1,...,y_n)$ and the standard symplectic form $\omega = \sum_j dx_j \wedge dy_j$. Consider also  the  Euclidean unit ball $\B$, and define the cylinder
$\Sigma = \{ (x,y): x_1^2 + y_1^2 < 1 \}$. Seminal Gromov's non-squeezing theorem \cite{Gr} states that if for some $r,R > 0$ there exists a symplectic  embedding $f: r\B \to R\Sigma$, that is,  $f^*\omega = \omega$, then $r \le R$.

Let $\H$ be a complex  Hilbert space with a fixed orthonormal basis $\{ e_n \}$ and the standard symplectic structure $\omega$. We now use the notation
\begin{eqnarray}
\label{coordinates2}
Z = (z,w) = (z,w_1,w_2,\ldots)
\end{eqnarray}
for the coordinates in $\H$. Here $z=\langle Z,e_1 \rangle\in\C$. For a domain $\Omega \subset \C$ we define  the cylinder $\Sigma_\Omega= \{ Z \in \H: z \in \Omega \}$ in $\H$. Let $\Phi$ be a symplectomorphism. As in Section 2, we use the notation
$$P = \Phi_Z \,\,\, \mbox{and} \,\,\, Q = \Phi_{\bar Z}.$$
Recall that in Section 2 we proved that the complex representation $A_J$ for  the almost complex structure $J=\Phi_*(J_{st})$ has the form
$$A_J = Q \bar{P}^{-1}.$$
We prove a version of non-squeezing theorem under the assumption that $A_J$ is small enough.
Our main result here is the following
\begin{thm}
\label{squiz}
(Non-squeezing theorem.) Let $r,R > 0$ and  $G$ be a domain  in $\Sigma_{R\D}$. There exists $\varepsilon_0 > 0$ with the following property: if  there exists a   symplectomorphism $\Phi: r\B^\infty \to G$  such that
\begin{eqnarray}
\label{NS1}
\| Q \bar{P}^{-1}\|_{C^1(r\B^\infty)} \le \varepsilon_0
\end{eqnarray}
then $r \leq R$.
\end{thm}

The condition (\ref{NS1}) means that the ``anti-holomorphic part'' of $\Phi$ is small enough. It holds if $\Phi$ is a small perturbation of a holomorphic symplectic map. In particular, the  assumption (\ref{NS1})   holds automatically if $\Phi$ and $\Phi^{-1}$ are close to the identity map in the $C^2$ norm on $r\B^\infty$ and $G$ respectively. In particular, this gives the non-squeezing theorem for short-time symplectic flows. The case of long-time symplectic flows is proved in \cite{SuTu4}. Essentially it is a consequence of  Theorem \ref{ThDiscs0}. It requires an additional assumption of regularity of a symplectic flow in Hilbert scales. Theorem \ref{squiz} shows that this regularity assumption can be dropped in the short-time case. This is due to the fact that the assumption (\ref{NS1}) allows us to use the implicit function theorem instead of Schauder's fixed point theorem.

Theorem \ref{squiz} is a consequence the following proposition concerning the existence of $J$-complex discs for $J = \Phi_*(J_{st})$.

\begin{prop}
\label{ThDiscs}
Under the assumptions of Theorem \ref{squiz}, for every point $(z^0,w^0) \in \Sigma_{R\D}$ there is a solution $Z\in C^{1,\alpha}(\D,\H_0)$, $0 < \alpha < 1$ of (\ref{holomorphy}) such that $Z(\bar\D)\subset\bar\Sigma_\D$, $(z^0,w^0) \in Z(\D)$, $\Area(Z) = \pi R^2$, and
\begin{eqnarray}
\label{BC1}
Z(b\D) \subset
b\Sigma.
\end{eqnarray}
\end{prop}

Let us prove Theorem \ref{squiz} assuming Proposition \ref{ThDiscs}. We  essentially follow the original argument of Gromov \cite{Gr}.
\smallskip

\noindent
{\bf Proof of Theorem \ref{squiz}.}
Since $\Phi^*\omega=\omega$, the almost complex structure $J := \Phi_*(J_{st}) = d\Phi \circ J_{st} \circ d\Phi^{-1}$ is tamed by $\omega$.
Then the complex representation $\tilde A$ of $J$ is defined by (\ref{A}).
Fix $\varepsilon > 0$.
Let $\chi$ be a smooth cut-off function with support in $G$
and such that $\chi=1$ on $\Phi((r-\varepsilon)\bar\B^\infty)$.
Define $A=\chi\tilde A$. Let $p = \Phi(0)$. By Proposition \ref{ThDiscs} there exists a solution $Z$ of (\ref{holomorphy}) such that $p \in Z(\D)$, $Z(b\D) \subset b\Sigma$ and $\Area(Z) = \pi R^2$. Note that this disc is smooth in $\D$ by Theorem \ref{InnerDiscRegularity}. Denote by $D \subset \D$ a connected component of the pull-image $Z^{-1}(\Phi((r-\varepsilon)\B^\infty))$.
Then $X = \Phi^{-1}(Z(D))$ is a closed $J_{st}$-complex curve  in $(r-\varepsilon)\B^\infty$
with boundary contained in $(r-\varepsilon)b\B^\infty$. Furthermore, $0 \in X$ and $\Area(X) \leq \pi R^2$.

Consider the canonical projection $\pi_n: \H \to \C^n$, $\pi_n: Z = (Z_1,Z_2,...) \mapsto (Z_1,Z_2,...,Z_n)$. Put $Z' = \Phi^{-1} \circ Z$. Fix $n$ big enough such that $(\sum_{j=1}^n | Z_j'(\zeta) |^2)^{1/2} > (1-2\varepsilon)r$ for every $\zeta \in bD$. Then
$X_n:= (\pi_n \circ \Phi^{-1} \circ Z )(D) \cap (r- 2 \varepsilon)\B^n$ is a closed complex (with respect to $J_{st}$) curve through the origin in $\B^n$.
By the classical result due to Lelong (see, e.g., \cite{Ch}) we have $\Area(X_n) \geq \pi (r-2\varepsilon)^2$. Since $\Area(X_n) \le \Area(X)$ and $\varepsilon$ is arbitrary, we have $r\le R$ as desired.
$\blacksquare$
\medskip

{\bf Proof of Proposition \ref{ThDiscs}.}
Without loss of generality assume $R = 1$. Use the notation $p = (p_1,p_2,...,p_n,...) = (p_1,p') \in \H$. Denote by $\H_\R$ the real span of $\{ e_j\}$. Denote by ${\mathcal M}$ the space of $C^2$ maps from $\H$ to ${\mathcal L}(\H)$.

In the case where $A = 0$, i.e., $J = J_{st}$  we have the family of $J_{st}$-holomorphic discs
\begin{eqnarray}
\label{NS3}
\zeta \mapsto Z^0(\zeta) = (z(\zeta),w(\zeta)) = (\zeta,p'),
\end{eqnarray}
which clearly satisfies Proposition \ref{ThDiscs} because $p = (p_1,p') \in Z^0(\D)$ for every $p_1 \in \D$.

 We will prove that for $A$ close to $0$ in the $C^1$-norm this family can be perturbed to a family of $J$-holomorphic ($A = A_J$) discs proving the proposition.

Let  $A \in {\mathcal M}$, $d \in \H_\R$ and let $Z: \D \ni \zeta \mapsto Z(\zeta) = (z(\zeta),w(\zeta))$ be a map of class $C^{1,\alpha}(\D,\H)$. Consider the map
$$
\Lambda: (A,Z,d) \mapsto (\Xi,\Theta,\Gamma),
$$
where
$$\Xi = Z_{\bar\zeta} - A(Z) \bar{Z}_{\bar\zeta}, \,\,\, \zeta \in \D,$$
$$\Theta = z\bar{z} - 1,\,\,\, \zeta \in b\D,$$
$$\Gamma = \Re w - d, \,\,\, \zeta \in b\D.$$
 We view $A$ as a parameter considering $\Lambda_A = \Lambda(A,\bullet,\bullet)$.
Thus
$$
\Lambda_A: C^{1,\alpha}(\D,\H)\times \H_\R \to C^{\alpha}(\D,\H) \times  C^{1,\alpha}(b\D) \times  C^{1,\alpha}(b\D,\H).
$$
Consider the map $\tilde A: Z \mapsto A(Z)\bar{Z^0}_{\bar\zeta}$. The Fr\'echet derivative $\dot\Lambda$ of $\Lambda_A$ at the point $Z^0$ is the map
$$\dot\Lambda_A: C^{1,\alpha}(\D,\H) \times \H_\R \to C^{\alpha}(\D,\H) \times C^{1,\alpha}(b\D) \times C^{1,\alpha}(b\D,\H)$$
given by
$$\dot\Lambda_A: (\dot Z,d) \mapsto (\dot Z_{\bar\zeta} - L_A\dot Z - A(Z^0) \bar{\dot{Z}_\zeta}, 2\Re \bar\zeta \dot z(\zeta)|_{b\D}, \Re \dot w(\zeta)|_{b\D} -d).$$
Here $L_A  = d\tilde A(Z^0): C^{1,\alpha}(\D,\H) \to C^{1,\alpha}(\D,\H)$ is an $\R$-linear bounded operator continuously depending on  $dA(Z^0)$; furthermore $L_A = 0$ when $dA(Z^0) = 0$.

For $A = 0$ we obtain the operator
$$L_0: (\dot Z,d) \mapsto (\dot Z_{\bar\zeta}, 2\Re \overline\zeta \dot z(\zeta)|_{b\D}, \Re \dot w(\zeta)|_{b\D}-d).$$
It is easy to see that this is a  bounded surjective  operator. Indeed, let $L_0(\dot Z,d) = (Z',g,h)$ with $Z' = (z',w')$.

This is a linear Riemann-Hilbert boundary value problem which splits into two independent boundary value problems for $\dot w$ and $\dot z$ respectively. They can be explicitly solved by the Cauchy integral (more precisely, by the generalized Cauchy-Schwarz integral). Since the index of the boundary value problem for $\dot w$ is equal to $0$, the solution is given by
\begin{eqnarray*}
\dot w(\zeta) = \frac{1}{2\pi i} \int_\D \left (\frac{w'(t)}{t-\zeta} + \frac{\zeta \bar{w'(t)}}{1-\bar{t}\zeta} \right ) d\zeta \wedge d\bar{\zeta} + \frac{1}{2\pi i} \int_{b\D} g(t) \frac{t+\zeta}{t-\zeta} \frac{dt}{t} +  d + ic_0
\end{eqnarray*}
with $c_0 \in \R$. The index of the boundary value problem for $\dot z$ is equal to $1$. Therefore the solution is given by
\begin{eqnarray*}
\dot z(\zeta) = \frac{1}{2\pi i} \int_\D \left (\frac{z'(t)}{t-\zeta} + \frac{\zeta^3 \bar{z'(t)}}{1-\bar{t}\zeta} \right ) d\zeta \wedge d\bar{\zeta} + \frac{\zeta}{2\pi i} \int_{b\D} h(t) \frac{t+\zeta}{t-\zeta} \frac{dt}{t} + c_0\zeta + c_1 \zeta + c_2 \zeta^2
\end{eqnarray*}
with $c_{2-k} = -\bar{c_k}$, see for example \cite{Ve}. In view of the regularity properties of the Cauchy integrals (Theorem \ref{VectorCauchyTm}) this proves that $\dot \Lambda_0$ is a bounded surjective operator so the same is true for $\dot \Lambda_A$ when $A$ is close to $0$ in the $C^1$ norm.  By the implicit function theorem, we obtain a family of discs which is a small perturbation of (\ref{NS3}). Clearly this family fills $\Sigma_\D$. Since the real part of the $w_j$'s component is constant on $b\D$, it follows by Stokes' formula that $\int_\D dw_j \wedge dw_j = 0$. Hence the area of every disc is equal to 1. The proof is complete. $\blacksquare$

\section{Appendix I: Morrey's embedding}

We prove here the following
\begin{thm}
\label{Morrey}
(Morrey's embedding.) There is a bounded inclusion
$$
W^{1,p}(\D,\H) \to C^\alpha(\D,\H),
$$
where $p > 2$ and $\alpha = (p-2)/p$.
\end{thm}
\proof Let $u \in W^{1,p}(\D,\H)$. Then $\partial u /\partial\bar\zeta \in L^p(\D,\H)$ (weak derivative). Then $v = T \partial u /\partial\bar\zeta \in C^\alpha(\D,\H)$. Hence the $\H$-valued function $h = u - v$ is holomorphic in $\D$. But $\bar{u} \in W^{1,p}(\D,\H)$ as well.
Hence $ u = v +h = \bar{v_1} + \bar{h_1}$, where $v,v_1 \in C^{\alpha}(\D,\H)$ and the vector functions $h$ and $h_1$ are holomorphic in $\D$. Then $h-\bar{h_1} = v_0 = -v + \bar{v_1} \in C^\alpha(\D,\H)$. Then $h = K v_0 + \bar{h_1(0)}$, where $K$ denotes the Cauchy type integral over $b\D$. We have
\begin{align*}
&\|Kv_0\|_{C^\alpha(\D,\H)} \le\const \|v_0\|_{C^\alpha(b\D,\H)}\le\const\|u\|_{W^{1,p}(\D,\H)},\\
&\|h_1(0)\| \le \const \|h_1\|_{L^p(\D,\H)}\le\const (\|u\|_{L^p(\D,\H)}+\|v_1\|_{C^\alpha(\D,\H)})\le\const\|u\|_{W^{1,p}(\D,\H)}.
\end{align*}
Hence, $\|u\|_{C^\alpha(\D,\H)}\le\const\|u\|_{W^{1,p}(\D,\H)}$.
The proof is complete.
$\blacksquare $

\section{Appendix II: Proof of Theorem \ref{TheoCauchy1}}

We use the notation
$$
d^2t:= \frac{dt \wedge d\bar{t}}{2 \pi i}.
$$
Consider  the operators
\[
T^Qf = Q T(f/Q)\quad
\text{and}\quad
S^Q f = \partial T^Q f.
\]
Here $S^Q f$ is defined for $f\in C_0^\infty(\D)$ as a pointwise derivative whenever it exists. As usual, $C_0^\infty(\D)$ denotes the space of complex smooth functions with compact support in $\D$.

\begin{lemma}
\label{LemCauchy1}
It suffices to show that $S^Q$ extends as a bounded operator $L^p(\D) \to L^p(2\D)$.
\end{lemma}
\begin{proof} Suppose that $S^Q: L^p(\D) \to L^p(2\D)$ is bounded.  We show that $S_Q:L^p(\D) \to L^p(\D)$ is bounded. We make all estimates for $f \in C^\infty_0(\D)$; the conclusion will follow by density of $C^\infty_0(\D)$ in $L^p(\D)$.

{\it Step 1.} We first show $S_Q:L^p(\D) \to L^p(\frac{1}{2}\D)$ is bounded.

Put $T_Q = T^Q + \tilde T$, $S_Q = S^Q + \tilde S$. Note that  $\tilde S = \partial \tilde T$ since $f \in C^\infty_0(\D)$.  The assumption  $p > p_1$ implies  $Q^{-1} \in L^q(\D)$, $1/p + 1/q = 1$. We have
$$
\tilde Tf(z) = Q(z) \int_\D \frac{\bar{f(t)}d^2t}{\bar{Q(t)}(\bar{t} z -1)}.
$$
If now $| z | \le 1/2$, then $|\bar{t} z -1| \ge 1/2$ and by H\"older inequality
$$
|\tilde S f(z)| \le \const \|f\|_p \|Q^{-1}\|_q,
$$
that is, $\tilde S:L^p(\D) \to L^\infty(\frac{1}{2}\D)$ is bounded. Hence $\tilde S: L^p(\D) \to L^\infty(\frac{1}{2}\D)$ is bounded.

{\it Step 2.} We now express $T_Q$ in terms of $T^Q$, namely,
\begin{eqnarray*}
T_Qf(z) = T^Q f(z) + Q(z) z^{-1} \bar{Q(\bar{z}^{-1})}\,^{-1} \bar{T^Qf(\bar{z}^{-1})}.
\end{eqnarray*}
We have
\begin{eqnarray*}
Q(z)\bar{Q(\bar{z}^{-1})}\,^{-1} = \prod_{k=1}^n (z-z_k)^{\alpha_k} \bar{ (\bar z\,^{-1}-z_k)}\,^{-\alpha_k} = \prod_{k=1}^n (-z z_k)^{\alpha_k} =\rho(z) z^{\sum\alpha_k},
\end{eqnarray*}
here $\rho(z)$ is a locally constant function in $\C \setminus \cup_k \Gamma_k$ and $| \rho(z) | = 1$.
For simplicity we use the assumption
\begin{eqnarray}
\label{Cauchy6}
\sum_{k=1}^n \alpha_k = 1,
\end{eqnarray}
which is the case of our application,  although the result holds without this restriction. We now have
$$T_Qf(z) = T^Qf(z) + \rho(z) \bar{T^Qf(\bar{z}^{-1})},$$
$$S_Qf(z) = S^Qf(z) - z^{-2} \rho(z) \bar{S^Qf(\bar{z}^{-1})}$$
We note that the pointwise and weak derivatives coincide for $f\in C^\infty_0(\D)$. Since $S^Q: L^p(\D) \to L^p(2\D)$ is bounded, it
 follows by the substitution $z \mapsto \bar{z}^{-1}$, that the operator $S_Q: L^p(\D) \to L^p(\D \setminus\frac{1}{2}\D)$ is bounded. Combining with Step 1, we conclude that  $S_Q:L^p(\D) \to L^p(\D)$ is bounded.
\end{proof} $\blacksquare$
\smallskip

Introduce
$$
\tilde S^Qf:= Q \, S(f/Q).
$$
Since $\partial Q = Q \sum_{k=1}^n \alpha_k (z - z_k)^{-1}$, one can see
$$
S^Q = \sum_{k=1}^n \alpha_k \tilde S_k + (1 - \sum_{k=1}^n \alpha_k) \tilde S^Q.
$$
Here $\tilde S_k:= \tilde S^{Q_k}$ and $Q_k(z) = (z-z_k)^{-1} Q(z)$. With our assumption (\ref{Cauchy6}), this gives
$$
S^Q = \sum_{k=1}^n \alpha_k \tilde S_k.
$$
Slightly changing notation, we now allow $Q$ to have negative powers, that is,
$$
Q(z) = \prod_{k=1}^{m} (z-z_k)^{\alpha_k} \prod_{k=m+1}^n (z-z_k)^{-\alpha_k},
$$
$0 < \alpha_k < 1$ for $1\le k \le n$. Our theorem is a consequence of the following result from \cite{Mo}.

\begin{thm}
\label{TheoCauchy2}
Let $p_1' < p < p_2'$, here
$$1 < p_1' = \max_{1 \le k \le m} \frac{2}{2-\alpha_k} < 2 < p_2' = \min_{k> m} \frac{2}{\alpha_k}.$$
Then the operator $\tilde S^Q: L^p(\D) \to L^p(2\D)$ is bounded.
\end{thm}
Note that for Theorem \ref{TheoCauchy1} we need $Q$ with only one negative factor $(z-z_k)^{-\alpha_k}$, in which $\alpha_k$ is the old $1-\alpha_k$.
\smallskip

\proof
We proceed in three steps.

{\it Step 1.} We first consider the case where
$$
Q(z) = (z-z_0)^{-\alpha}, \; | z_0 | = 1, \; 0 < \alpha < 1.
$$
Then $\tilde S^Q = S + \tilde S$ with
$$
\tilde S f(z) = (z-z_0)^{-\alpha}\int_\D f(t) \frac{(t-z_0)^\alpha - (z-z_0)^\alpha}{(t-z)^2} d^2t.
$$
Since $| (t-z_0)^\alpha - (z - z_0)^\alpha | \le | t - z |^\alpha$,
we have
$$
|\tilde S f(z) | \le | z - z_0 |^{-\alpha}\int_\D | f(t) | | t - z |^{\alpha - 2} |d^2t |.
$$
Introduce $\gamma > 0$, which we will choose later. By the H\"older inequality for the measure $| t - z |^{\alpha-2}| d^2 t |$ we obtain
\begin{align*}
|\tilde S f(z)| &\le |z-z_0|^{-\alpha} \left(\int_\D |f(t)|^p |t-z_0 |^{\gamma p}|t-z|^{\alpha-2} |d^2 t|\right)^{1/p}\\
& \; \times \left(\int_\D  |t-z_0|^{-\gamma q}|t-z|^{\alpha-2}|d^2 t| \right)^{1/q}.
\end{align*}
We use the classical estimate
\begin{eqnarray}
\label{Cauchy7}
J(\alpha,\beta)  = \int_\D | t- z_0 |^{-\alpha} | t - z |^{-\beta} | d^2t| \le M_{\alpha\beta} | z - z_0 |^{2 -\alpha - \beta}
\end{eqnarray}
here $0 < \alpha < 2$, $0 < \beta < 2$, $\alpha + \beta > 2$ (see \cite{Ve}, proof of Theorem 1.19). Suppose that
\begin{eqnarray}
\label{Cauchy8}
\gamma q < 2, \quad \gamma q -\alpha > 0
\end{eqnarray}
For $\sigma = -\alpha p
+ (\alpha - \gamma q) p/q = -\alpha - \gamma p$ we have
$$
| \tilde S f(z) |^p \le \const | z - z_0 |^\sigma \int_\D | f(t) |^p | t - z_0 |^{\gamma p} | t - z |^{\alpha - 2} | d^2 t |
$$
and
$$
\| \tilde S f  \|^p_{L^p(2\D)} \le \const \int_{2\D} \left ( \int_\D | f(t) |^p | t - z_0 |^{\gamma p}| t - z |^{\alpha-2} | z - z_0 |^\sigma |d^2t| \right ) |d^2z|.
$$
Suppose now that
\begin{eqnarray}
\label{Cauchy9}
\alpha + \gamma p = -\sigma < 2.
\end{eqnarray}
Integrating first with respect to $z$ and using (\ref{Cauchy7}) we obtain
$$
\| \tilde Sf \|_p^p \le \const \int_\D | f(t) |^p | t - z_0 |^{\gamma p} | t - z_0 |^{\alpha + \sigma} | d^2 t | = \const \| f \|_p^p.
$$
We now show that $\gamma > 0$ satisfying (\ref{Cauchy8}, \ref{Cauchy9}) does exist, that is,
\begin{align*}
&\gamma < \gamma_1 = 2/q = 2(p-1)/p,\\
&\gamma < \gamma_2 = (2 - \alpha)/p,\\
&\gamma > \gamma_3 = \alpha/q = \alpha (p-1)/p.
\end{align*}
Obviously $\gamma_3 < \gamma_1$.
Since $p < p_2'$, we have $\alpha p<2$. Hence $\gamma_3 < \gamma_2$, and the desired $\gamma$ exists.
\smallskip

{\bf Remark.} In \cite{Mo} $\gamma = \alpha/p$ which works only if $1 + \alpha/2 < p < 2$.
\smallskip

{\it Step 2.} We now consider another special case
$$Q(z) = (z - z_0)^\alpha, \,\,\,| z_0 | = 1, \,\,\, 0 < \alpha < 1.$$
In this case the analysis is similar, even simpler. Again put $\tilde S^Q = S + \tilde S$ with
$$\tilde S f(z) = \int_\D f(t) \frac{(z-z_0)^\alpha - (t-z_0)^\alpha}{(t-z_0)^\alpha(t-z)^2}d^2t.$$
Then for some $\beta > 0$ using the H\"older inequality
\begin{align*}
| \tilde S f(z) | &\le \int_\D | f(t) | | t - z_0|^{-\alpha}| t-z|^{\alpha-2}| d^2t| \\
&\le \left ( \int_\D | f(t) |^p | t- z_0 |^{\beta p}| t-z|^{\alpha-2}| d^2t| \right )^{1/p}\left ( \int_\D | t- z_0 |^{-(\alpha +\beta) q}| t-z|^{\alpha-2}| d^2t| \right )^{1/q}.
\end{align*}
Using (\ref{Cauchy7}) again
$$
| \tilde S f(z) |^p \le \const | z - z_0|^\sigma \int_\D | f(t) |^p | t- z_0 |^{\beta p}| t-z|^{\alpha-2}| d^2t|,
$$
where $\sigma = -\alpha - \beta p$ provided that
\begin{eqnarray}
\label{Cauchy10}
 (\alpha + \beta)q < 2.
\end{eqnarray}
Integrating first with respect to $z$ and using (\ref{Cauchy7}) we obtain
\begin{align*}
\| \tilde S f \|_{L^p(2\D)}^p &\le \const \int_{2\D}\left (  \int_\D | f(t) |^p | t-z_0|^{\beta p}
| t - z |^{\alpha-2}| z - z_0|^\sigma | d^2t| \right ) | d^2 z|\\
&\le \const \int_\D | f(t) |^p | t - z_0 |^{\beta p} | t- z_0 |^{\alpha + \sigma} | d^2t| =\const \| f \|_p^p
\end{align*}
provided that
\begin{eqnarray}
\label{Cauchy11}
\alpha + \beta p = -\sigma < 2.
\end{eqnarray}
Since $p > p_1'$, we have $p > 2/(2-\alpha)$ and $\alpha q < 2$. Therefore, for sufficiently small $\beta > 0$ both (\ref{Cauchy10}) and (\ref{Cauchy11}) are satisfied.
\smallskip

{\bf Remark.} In \cite{Mo} $\beta = [(2-\alpha)p - 2]/p^2$, which satisfies both (\ref{Cauchy10}), (\ref{Cauchy11}), but the particular choice is unimportant.
\smallskip

{\it Step 3.} Finally, consider the general case. Let $U_k = \{ z \in \C: | z - z_k | < \delta \}$, $1 \le k \le n$, $\delta > 0$ is small enough so that the closed discs $\bar{U_k}$ are disjoint.
Put $U_0 = \D \setminus \cup_k U_k$. Then
$$
\tilde S^Q f = \sum_{k=0}^n S_k f, \qquad
S_k f(z) = Q(z) \int_{U_k \cap\D} \frac{f(t) d^2t}{Q(t)(t-z)^2}.
$$
We claim that every $S_k:L^p(\D) \to L^p(2\D)$ is bounded. For definiteness choose $1 \le k \le m$, for other
$k$ the analysis is similar. Let $Q(z) = \tilde Q(z) (z-z_k)^{\alpha_k}$. Define
\begin{equation*}
f_k(z)=\begin{cases}
f(z)/\tilde Q(z)& \text{if}\quad z\in U_k\cap\D, \\
0 & \text{otherwise}.
\end{cases}
\end{equation*}
Then $f_k \in L^p(\D)$. Introduce the function
$$
g_k(z) = (z-z_k)^{\alpha_k}\int_\D \frac{f_k(t)d^2t}{(t-z_k)^{\alpha_k}(t-z)^2}.
$$
By one of the two special cases considered before, $g_k \in L^p(2\D)$ and $S_k f= \tilde Q \, g_k$.

Note that $\tilde Q$ is bounded in a neighborhood of $\bar{U_k}$, and
$p < p_2'$ implies $\tilde Q \in L^p(2\D)$.
Every $g_k$ is bounded in a neighborhood of each $z_j$, $j \neq k$.
Moreover,
$$
\| g_k \|_{L^\infty(U_j)} \le \const \| f \|_{L^p(\D)}\quad
\text{and}\quad
\| g_k \|_{L^p(2\D)} \le \const \| f \|_{L^p(\D)}.
$$
Hence $\| S_k f \|_{L^p(2\D)} \le \const \| f \|_{L^p(\D)}$, which completes the proof of the theorem.

\end{document}